\documentclass[12pt]{amsart}

\usepackage{amssymb,latexsym,amscd,amsfonts,subcaption}

\usepackage{amssymb,latexsym,amscd,amsfonts,subcaption,relsize}
\usepackage{hyperref}
\usepackage{fixltx2e}
\DeclareFontFamily{U}{BOONDOX-calo}{\skewchar\font=45 }
\DeclareFontShape{U}{BOONDOX-calo}{m}{n}{
  <-> s*[1.05] BOONDOX-r-calo}{}
\DeclareFontShape{U}{BOONDOX-calo}{b}{n}{
  <-> s*[1.05] BOONDOX-b-calo}{}
\DeclareMathAlphabet{\mathcalboondox}{U}{BOONDOX-calo}{m}{n}
\SetMathAlphabet{\mathcalboondox}{bold}{U}{BOONDOX-calo}{b}{n}
\DeclareMathAlphabet{\mathbcalboondox}{U}{BOONDOX-calo}{b}{n}
\usepackage[utf8]{inputenc}
\usepackage[T1]{fontenc}
\usepackage{doi}
\usepackage[normalem]{ulem}
\usepackage{tikz}
\usepackage{calc}
\usetikzlibrary{calc}
\usetikzlibrary{positioning}

  \makeatletter
\CheckCommand*\refstepcounter[1]{\stepcounter{#1}%
    \protected@edef\@currentlabel
       {\csname p@#1\endcsname\csname the#1\endcsname}%
  }
  \renewcommand*\refstepcounter[1]{\stepcounter{#1}%
    \protected@edef\@currentlabel
      {\csname p@#1\expandafter\endcsname\csname the#1\endcsname}%
  }
  \def\labelformat#1{\expandafter\def\csname p@#1\endcsname##1}
  \DeclareRobustCommand\Ref[1]{\protected@edef\@tempa{\ref{#1}}%
     \expandafter\MakeUppercase\@tempa
  }
  \makeatother
  
  \def\claimproofname{{\noindent \textsf{Proof of Claim:}}}

  \makeatletter
  \newcommand{\numberlike}[2]{%
     \expandafter\def\csname c@#1\endcsname{%
         \expandafter\csname c@#2\endcsname}%
  }
  \makeatother

 \def\DefaultNumberTheoremWithin{section}
  \theoremstyle{plain}
  \newtheorem{lem}{Lemma}
     \numberwithin{lem}{\DefaultNumberTheoremWithin}
     \labelformat{lem}{Lemma~#1}
  
     \numberwithin{Claim}{\DefaultNumberTheoremWithin}
     \numberlike{Claim}{lem}
     \labelformat{Claim}{Claim~#1}
  \newtheorem{thm}{Theorem}
     \numberwithin{thm}{\DefaultNumberTheoremWithin}
     \numberlike{thm}{lem}
     \labelformat{thm}{Theorem~#1}
  \newtheorem{cor}{Corollary}
     \numberwithin{cor}{\DefaultNumberTheoremWithin}
     \numberlike{cor}{lem}
     \labelformat{cor}{Corollary~#1}
  \newtheorem{prop}{Proposition}
     \numberwithin{prop}{\DefaultNumberTheoremWithin}
     \numberlike{prop}{lem}
     \labelformat{prop}{Proposition~#1}

  \newtheorem{Conjecture}{Conjecture}
     \numberwithin{Conjecture}{\DefaultNumberTheoremWithin}
     \numberlike{Conjecture}{lem}
     \labelformat{Conjecture}{Conjecture~#1}

  \theoremstyle{definition}
  \newtheorem{Definition}{Definition}
     \numberwithin{Definition}{\DefaultNumberTheoremWithin}
     \numberlike{Definition}{lem}
     \labelformat{Definition}{Definition~#1}

  \theoremstyle{definition}
  
     \numberwithin{Question}{\DefaultNumberTheoremWithin}
     \numberlike{Question}{lem}
     \labelformat{Question}{Question~#1}

  \theoremstyle{definition}
  
     \numberwithin{Problem}{\DefaultNumberTheoremWithin}
     \numberlike{Problem}{lem}
     \labelformat{Problem}{Problem~#1}

  \theoremstyle{remark}
  
     \numberwithin{Remark}{\DefaultNumberTheoremWithin}
     \numberlike{Remark}{lem}
     \labelformat{Remark}{Remark~#1}

  \newtheorem{Example}{Example}
     \numberwithin{Example}{\DefaultNumberTheoremWithin}
     \numberlike{Example}{lem}
     \labelformat{Example}{Example~#1}
  
     \labelformat{Case}{Case~#1}
     \numberwithin{Case}{lem}
  
     \labelformat{Step}{Step~#1}
     \numberwithin{Step}{lem}

  \labelformat{equation}{(#1)}
  \labelformat{figure}{Figure~#1}
  \labelformat{section}{\S#1}
  \labelformat{subsection}{\S#1}
  \labelformat{subsubsection}{\S#1}

  \def\eqref{\ref}

\def\hpic #1 #2 {\mbox{$\begin{array}[c]{l} \epsfig{file=#1,height=#2} \end{array}$}}
\def\vpic #1 #2 {\mbox{$\begin{array}[c]{l} \epsfig{file=#1,width=#2} \end{array}$}}
\def\ignore #1 {}

\def\Spec{{\mathrm{Spec}}}

\def\fa{{\mathfrak{a}}}
\def\fp{{\mathfrak{p}}}
\def\fb{\mathfrak{b}}
\def\ca{\mathcalboondox{a}}
\def\cb{\mathcalboondox{b}}

\def\ima{\mathrm{im}}
\def\Cov{\mathrm{Cov}}
\def\link{\mathrm{link}}

\def\gposet{{\mathcal{S}}}
\def\laplace{\mathcal{L}}

\def\Hom{\mathrm{H}}

\def\RR{\mbox{{$\mathbb{R}$}}}

\def\QQ{\mbox{{$\mathbb{Q}$}}}

\def\TT{\mbox{${\mathbb T}$}}

\def\a{\alpha}

\def\Pr{\mathrm{Pr}}
\def\pr{\mathrm{pr}}

\newcommand{\1}{{\mathbf 1}}

\renewcommand{\t}{\tau}

\def \aaron #1 {\marginpar{#1 -AA}}

\begin{document}

\author[E. Babson]{Eric Babson}
\address{1094 Cragmont Avenue, Berkeley, CA 94708, USA}
\email{onemailuser@proton.me}
\author[V. Welker]{Volkmar Welker}
\address{Philipps-Universit\"at Marburg \\
Fachbereich Mathematik und Informatik \\
35032 Marburg \\ Germany}
\email{welker@mathematik.uni-marburg.de}
\title[Versions of the Garland method]{Homological algebra and poset versions of the Garland method}
\thanks{This material is based upon work supported by the National Science Foundation under Grant No. DMS-1440140 while the second author was in 
residence at the
Mathematical Sciences Research Institute in Berkeley, California, USA}

\begin{abstract}
  Garland introduced a vanishing criterion for characteristic zero cohomology groups of locally finite and locally connected simplicial complexes based on the spectral gaps of the graph Laplacians of face links which has turned out to be effective in a wide range of examples.
  In this work we provide a homological algebra version of this method and define a class of posets which provide examples including Garland's original vanishing theorem in the simplicial setting.  The cubical case is analyzed in some detail where it provides a theorem about the structure rather than vanishing of cohomology.  Finally a random cubical setting is contrasted with a familiar simplicial one.  
\end{abstract}

\maketitle

\section{Introduction}

In \cite{Gar} Garland introduced a local to global method to establish characteristic zero cohomological vanishing for locally finite locally connected 
simplicial complexes from sufficiently large spectral gaps for the Laplacians of link graphs.   This has turned out to be effective in a wide range of examples in fields including geometric group theory (see e.g., \cite{BS,Z,Op}) and geometric combinatorics (see e.g., \cite{GW,HKP}).

In this note we generalize this approach to an exactness criterion for pairs of $3$-term complexes of real Hilbert-spaces in \ref{thm:garland}.  
In Section \ref{sec:6} we expand on examples of simplicial and cubical complexes where the global property becomes a geometric restriction on cohomology 
classes and the local one involves both (in the cubical case) 
geometric link graphs and transversal graphs. We then analyze models for
random simplicial and cubical complexes which turn out to have different 
cohomologcal behavior. 

A related albeit different abstract approach to Garland's method \cite{L} appeared 
after this paper.
Garland posets also appear under a different name in the theory of high dimensional expanders (see e.g., \cite{KT}).

Garland structures and Garland posets are introduced in \ref{sec:2} and 
\ref{sec:3} respectively.  These give a homological algebra setting for the method and a collection of combinatorial examples (\ref{prop:garlandstructure}) including locally finite simplicial (\ref{cor:simplicialgarland}) and cubical (\ref{cor:cubicalgarland}) complexes which are expanded on in \ref{sec:4}.   Details of proofs are in \ref{sec:5}. In and \ref{sec:6} we consider several
classes of examples. Moment angle complexes serve as examples for the cubical case. For the cubical and the simplicial case we define models of
random complexes and study them.

\section{Garland Structures}
\label{sec:2}

Write $\fb^*$ for the adjoint of a (bounded linear) map $\fb$ of (real) Hilbert spaces and 
$\fp_X$ for the 
orthogonal projection to a (closed) subspace $X$.  
We call a pair $(A,\fa)$ a $3$-term complex if $$A: A^{-1}\overset{\fa_0}\rightarrow A^0\overset{\fa_1}\rightarrow A^1$$ with $\fa_0$ and $\fa_1$ maps of Hilbert spaces such that $\fa_1\fa_0 = 0$. 
The map $\laplace_A^+=\fa_1^*\fa_1$ is the positive
Laplacian of $(A,\fa)$,
$\Hom^0(A)=\frac{\ker(\fa_1)}{\ima(\fa_0)}\cong \ker(\fa^*_0)\cap \ker(\fa_1)$ is its homology and the complex is called {\sf exact} if this homology vanishes.

\begin{Definition}
  A {\sf Garland structure} $G=(A,B,\fa)$ is a pair of $3$-term complexes $(A,\fa)$ and $(B,\fb)$ with $A$ exact, $B^i\subseteq A^i$, $B^{-1}=A^{-1}$ and 
	$\fb_i=\fp_{B^i}\fa_i$.  
\end{Definition}
Call $A$ and $B$ the local and global complexes of $G$ respectively.  

If $A : A^{-1} \xrightarrow{\fa_0} A^0 \xrightarrow{\fa_1} A^1$ is a three term complex set 
$$\alpha_A=\inf_{\genfrac{}{}{0pt}{}{\phi \in \ker(\fa^*_0)}{\|\phi\|=1}}\|\fa_1(\phi)\|^2.$$
Note that if $A$ is exact and $A^{-1}$ is nonzero then $\alpha_A$ is the spectral gap of the local positive Laplacian $\laplace_A^+$ which has spectrum contained in $\{0\}\cup[\alpha_A,\infty)$ and containing $0$ and $\alpha_A$. 

If $G = (A,B,\fa)$ is a Garland structure set
$$\alpha_G = \alpha_A \,\,\,\hbox{ and }\,\,\,
\beta_G=\sup_{\genfrac{}{}{0pt}{}{\phi \in \ker(\fb_1)}{\|\phi\|=1}}\|\fa_1(\phi)\|^2.$$  
As usual the infimum and supremum over the empty set are $\pm\infty$.  

\begin{thm} \label{thm:garland} 
  If a Garland structure has $\beta_G < \alpha_G$ then $B$ is exact.  
\end{thm}
\begin{proof}
  If $\Hom^0(B)\neq \{0\}$ there is $\phi\in \ker(\fb^*_0)\cap \ker(\fb_1)$ 
	with 
  $\|\phi\|=1$. Since $B^{-1}=A^{-1}$ 
  we have $\ker(\fb^*_0)\subseteq \ker(\fa^*_0)$. It then follows that  
  $\alpha_G\leq\|\fa_1(\phi)\|^2\leq\beta_G < \alpha_G$. 
\end{proof}

Garland structures form a category in which a map is a chain map between local complexes which induces one between the global subcomplexes.

In this category, if the local complex $(A,\fa)$ of a Garland structure $G$ 
has a direct summand $(A_+,\fa_+)$ then $G_+=(A_+,A_+\cap B, \fa_+)$ is a 
substructure of $G$.  

Note that if 
$$A^{-1}_-=\{0\}, ~~~~ A^0_-=\big\langle\,\{\,\phi \in \ker(\fa^*_0)\,:\,\|\phi\|=1,\,\|\fa_1(\phi)\|^2=\alpha_A\,\}\,\big\rangle$$ 
$$ \text{ and } A^1_-=\fa_1 A^0_-$$ then there is $A_+$ with $A=A_-\oplus A_+$.  Repeating yields:

\begin{cor}
	\label{cor:spectral}
The dimension of $\Hom^0(B)$ is at most that of the spectral projection of 
$\laplace_A^+$ to $[\alpha_G,\beta_G]$.
\end{cor}

Note that if a group acts by automorphisms on a Garland structure $G$ then it 
induces decompositions into direct sums of Garland structures and the theorem 
can be applied to each summand separately.  
An example where this approach may find applications is
the complex studied in \cite[Theorem 1.4]{PW}. There the authors consider 
the $n$-dimensional vector
space over the field with $q^2$ elements equipped with a (non-degenerate) 
unitary form. 
They consider the simplicial complex of collections of mutually orthogonal
non-degenerate $1$-dimensional subspaces. If $q$ is large compared to $n$
then they use the usual Garland method the show that the homology of 
the complex is concentrated in dimension $n-2$. For $q$ small compared to 
$n$ there are counterexamples and the homological structure is unknown. 
The group $GU_n(q)$ acts on this complex. Therefore, despite the fact that 
many technicalities will be involved, the separate
analysis of the corresponding ''irreducible'' Garland structures provides
a plausible approach.


\section{Garland Posets}
\label{sec:3}

In this section we define Garland posets to which we associate Garland structures.
In \ref{sec:4} these are shown to include posets associated to simplicial and cubical complexes for which the above theorem gives cohomological restrictions.
Proofs are sketched here and expanded in \ref{sec:proofs}.

For a poset $P$ and a subposet $Q$ we write $Q_{> x}$ ($Q_{< x}$) for the 
subposet of elements of $Q$ greater (less) than $x\in P$.  
Analogously defined are $Q_{\leq p}$ and $Q_{\geq p}$. 
An order relation $x < y$ in $P$ is called a cover relation if there is no
$z \in P$ with $x < z < y$. By $\Cov(P) \subseteq P \times P$ we denote the
set of all $(x,y)$ such that $x < y$ is a cover relation.


\begin{Definition}
  A {\sf Garland poset} $\gposet =(S,\leq,w,n,\rho)$ is a poset $(S,\leq)$ 
  together with a (cover preserving)  poset map $\rho : S \rightarrow \{ -1 < 0 < 1\}$,
  an orientation 
  $w:\Cov(S) \rightarrow \{-1,+1\}$, $(x,y) \mapsto w_{xy}$, and 
  numbers $n_0$, $n_1$ and $n_{010}$ so that for $S^i = \rho^{-1}(i)$
  we have:
  \begin{itemize}
    \item[(P1)] if $b\in S^0$ then $|S_{<b}|=n_0$,
    \item[(P2)] if $c\in S^1$ then $|S_{<c}^{-1}|=n_1$,
    \item[(P3)] if $S^0\ni b<c>b'\in S^0-\{b\}$ then $|S_{<b}\cap S_{<b'}|=n_{010}$,
    \item[(P4)] if $a\in S^{-1}$ then $S_{>a}$ is finite and connected,
    \item[(P5)] if $S^{-1}\ni a<c\in S^1$ then $S_{>a}\cap S_{<c}=\{b,b'\}$ with 
      \begin{align} \label{eq:w} 
        w_{ab}w_{bc}& =-w_{ab'}w_{b'c}.
      \end{align}
  \end{itemize}
\end{Definition}


In our examples of Garland posets we usually 
define the mutually disjoint sets $S^i$ directly, rendering $\rho$ superfluous.
Given a Garland poset $\gposet$ and a choice of $i$ consider the Hilbert space 
with orthonormal basis $$\{\,z_{axc}\,|\,S^{-1}\ni a\leq x\leq c\in S^1, 
x\in S^i\}.$$ Note that for $i=-1$ we have $z_{aac}$ as basis elements 
and for $i =1$ we have $z_{acc}$ as basis elements.
By (P4) the closed subspaces $A^i$ and $B^i$ generated 
respectively by 
$$\left\{z_{ax}=\sum_{c\in S_{\geq x}^1}z_{axc}\,|\,S^{-1}\ni a\leq x\in S^i\,\right\}\,\,\text{and}\,\,\left\{\,z_{x}=\sum_{a\in S_{\leq x}^{-1}}z_{ax}\,|\,x\in S^i\right\}$$ 
are well defined. 
The definitions immediately imply that $A^{-1} = B^{-1}$.
Note that the given generating sets are orthogonal but not normalized bases of the respective spaces.

\begin{lem}
	\label{lem:induced}
  The maps $$\fa_i\big(\,z_{ax}\,\big)=\sum_{y\in S_{>x}^{i}}w_{xy}z_{ay}$$ 
  extend to complexes 
	$$A_{\gposet} ~:~A^{-1} \xrightarrow{\fa_0} A^0 \xrightarrow{\fa_1} A^1$$
  and
	$$B_{\gposet} ~:~B^{-1} \xrightarrow{\fb_0 = \fp_{B^0}\fa_0} B^0 \xrightarrow{\fb_1 = \fp_{B^1} \fa_1} B^1$$
  of bounded linear maps.
\end{lem}

The proof of this is a simple calculation and can be found in \ref{sec:proofs}.
For $a \in S^{-1}$ we write $A_a$ for the three term subcomplex of $(A,\fa)$ with
$A^i_a$ generated by $\{\,z_{ax}\,|\,x\in S_{\geq a}^i\,\}$.
Note that by (P4) $A_a$ is of finite dimension. 
The following lemma is a direct consequence of our construction and again is proved in \ref{sec:proofs}. 

\begin{lem} \label{lem:asum} 
  If $\gposet$ is a Garland poset then 
  $A_{\gposet}=\overline{\oplus_{a\in S^{-1}}A_a}$ is the closure of the internal 
  orthogonal sum of complexes. 
	Moreover, each $A_a$ is exact and hence $A_{\gposet}$ is exact.
\end{lem}

By (P5) for a Garland poset $\gposet$ and $a \in S^ {-1}$ we can consider $S_{> a}$ as a graph, 
called the link graph at $a$,
with vertices $S^0_{> a}$ and edges $S^ 1_{> a}$. The disjoint 
union $\Gamma_{\gposet}$ of the graphs $S_{> a}$ for $a \in S^{-1}$ is called the link graph of $\gposet$. 
The positive Laplacian $\laplace^+_{A_a}$ of the complex $A_a$ is then the left normalized graph Laplacian 
(in the sense of \cite{J})
of the link graph at $a$. The normalization is induced by the fact that
the basis elements $z_{ab}$ of $A_a^{0}$ have norm the degree of $b$ in the
link graph at $a$. 
Note that if $\alpha_{A_a}\not=0$ the spectrum of $\laplace_a^+$ contains 
$0$ with multiplicity one and $\alpha_{A_a}$ and is contained in $\{0\}\cup [\alpha_{A_a},2]$. In particular, $\alpha_{A_a}$ is the spectral gap of
$\laplace_a^ +$. 
In the formulation of the following proposition we call a Garland structure $H = (A_H,B_H,\fa^H)$ a substructure of the
Garland structure $G = (A,B,\fa)$ if $A_H$ is a subcomplex of $A$, $B_H$ a subcomplex of $B$ and $\fa^H$ is the restriction of $\fa$. 

\begin{prop} \label{prop:garlandstructure} 
  If $\gposet$ is a Garland poset then 
  $G = (A_{\gposet},B_{\gposet},\fa)$ is a Garland structure with:
  \begin{itemize}
	  \item[(i)] $\alpha_{G}=\displaystyle{\inf_{a\in S^{-1}} \alpha_{A_a}}$ and
	  \item[(ii)] $\beta_{H} \leq \frac{n_0-n_{010}}{n_0}$
            if $H$ is any Garland substructure of $G$ with $\ker \fb_1 \cap B_H^0$ nonzero.
		  
      \end{itemize}
	In particular, $B_{\gposet}$ is acyclic if $\frac{n_0-n_{010}}{n_0}<\alpha_{A_a}$ for every $a\in S^{-1}$.   
\end{prop}

\begin{proof}[Sketch of proof (For a full proof see \ref{sec:5}):]
	The fact that $A_{\gposet}$ and $B_{\gposet}$ are complexes is proved in 
  \ref{lem:induced}.
	By \ref{lem:asum} $A_{\gposet}$ is exact. 
  Assertion (i) again follows from \ref{lem:asum}. 

  For (ii) note that if $\ker \fb_1^H \neq 0$ then 
	\begin{align}
		\label{eq:beta}
	\beta_{H} & \leq \sup_{0 \neq \phi\in B_H^0}\frac{\|\fa_1^H(\phi)\|^2-\kappa\|\fb_1^H(\phi)\|^2}{\|\phi\|^2} \end{align}
	for any 
	$\kappa >0$ since $\beta_{H}$ is a supremum over the kernel of $\fb_1^H$.   
	Choose $\kappa=\frac{n_{010}n_1}{n_0^2}$ and show that for every choice of $\phi$ the expression after the supremum is $\frac{n_0-n_{010}}{n_0}$. 

  The final assertion is a direct consequence of (i), (ii), \ref{thm:garland} and the previous paragraph.
\end{proof}

\section{Application to Simplicial and Cubical Complexes}
\label{sec:4}

Garland's original work treated the case of locally finite simplicial 
complexes with connected links and \ref{thm:garland} recovers his
vanishing condition in this case.
Our approach, among other examples, also treats 
locally finite cubical complexes and disconnected links in which cases
\ref{thm:garland} yields geometric and rank restrictions on generators
of cohomology rather than vanishing.  

\subsection{Simplicial}
\label{sec:simplicial}

Let $P$ be a locally finite simplicial poset.  This means that every lower interval is Boolean and every upper interval is finite. 
Equivalently, $P$ is the face poset of a locally finite $\Delta$-complex 
$\Delta_P$ (see \cite{H}).  
We write $P^k$ for the elements of $P$ covering $k+1$ elements and $P^K=\cup_{k\in K}P^k$ for the rank selected subposet of $P$.  

For each dimension $k\geq 1$ we construct a Garland poset $\gposet_kP = 
(S,\leq,w,n,\rho)$ 
with $S^{\{0,1\}}=P^{\{k,k+1\}}$ and $S^{-1}=\{\left(a,[b]\right)|P^{k-1}\ni a<b\in P^k\}$
with $[b]\subseteq P^k$ the vertex set of a connected component of
$P_{>a}^{\{k+1,k\}}$ containing $b$. In particular, 
$S^{-1}$ is in obvious bijection to $P^{k-1}$ in the 
case of connected links. We set 
$(a,[b])<b$ so $\pi:S\rightarrow P$ with $\pi(a,[b])=a$ and 
$\pi(c)=c$ for $c \in S^{\{0,1\}}$ is order preserving.  

If $n_0=k+1$, $n_1=\binom{k+2}{2}$ and $n_{010}=1$ 
and $\rho$ the obvious choice then $S$ clearly satisfies 
(P1), (P2), (P3) and (P4) in the definition of a Garland poset.

Fix a total ordering of $P^0$ which induces an orientation on $P^{\{k-1,k,k+1\}}$ with 
$u_{ab}=(-1)^r$ if $P^0\cap P_{< b}=\{v_0,v_1,\ldots v_s\} \ni v_r$ is listed in the fixed order and 
$v_r\not\leq a$.  
Use $\pi$ to pull the orientation back to $S$ with $w_{xy}=u_{\pi(x)\pi(y)}$.  Now (P5) clearly holds and $\gposet_kP=(S,\leq,w,n,\rho)$ is a Garland poset.  


Write $C^{k-1,k,k+1}$ for the degree $k-1$, $k$, and $k+1$ three term 
simplicial cochain complex for $\Delta_P$ with real coefficients determined 
by the given local ordering and note that 
$\pi^*:C^{k-1,k,k+1} \rightarrow B_{\gposet_kP}$
is an injective cochain map with cokernel $B^{-1}_{\gposet_kP} / C^{k-1}$ supported only in one degree and isomorphic to
$\bigoplus_{p\in P^{k-1}}\widetilde{\Hom}^0(\,\Delta_{P_{\geq p}}\,)$. 
For $p\in P^{k-1}$ write $\lambda_p: \widetilde{\Hom}^0(\,\Delta_{P_{\geq p}}\,) \rightarrow {\Hom}^k(\Delta_P)$  
for the connecting homomorphism in the induced long exact sequence 
$$\bigoplus_{p \in P^ {k-1}} \widetilde{\Hom}^0(\Delta_{P_{\geq p}}) \rightarrow \widetilde{\Hom}^k(\Delta_P) \rightarrow {\Hom}^0(B_{\gposet_kP}) \rightarrow 0.$$

Write $L_{k,P}$ for the span of the images of the $\lambda_p$ for all $p$.   
Thus by \ref{prop:garlandstructure} and \ref{thm:garland}:

\begin{cor} 
  \label{cor:simplicialgarland}
  If $P$ is a locally finite simplicial poset and the spectral gap 
  of every connected component of the link graph $\Gamma_{\gposet_kP}$ is 
	greater than $\frac{k}{k+1}$ then $\widetilde{\Hom}^k(\Delta_P)=L_{k,P}$.  
\end{cor} 

If all $P_{> p}$ for $p \in P^{k-1}$ are connected and $P$ is the face poset 
of a locally finite simplicial complex then $L_{k,P} =0$ and this is 
the original 
Garland result. 
Note that the gap condition for an $a\in S^{-1}$ means that the spectrum of $\laplace^+_{A_a}$ is contained in $\{0\}\cup (\frac{k}{k+1},2]$.  
Also note that if the link graph is regular of degree $d$ and has diameter 
at least four then the Alon-Boppona bound (see \cite{Lub,MarSpi}) 
implies that  this spectral bound can not be satisfied unless $d\geq k^2$.  

\subsection{Cubical}
\label{sec:cubical}

Let $P$ be a locally finite cubical poset.  This means that every lower interval is 
isomorphic to a product of copies of the three element poset 
$$
\begin{tikzpicture}[scale=1, every node/.style={font=\large}]
  \node (lambda) at (-2,0.7) {$\bigwedge = $};
  \node (minus) at (-1,0) {$-$};
  \node (plus)  at ( 1,0) {$+$};
  \node (star)  at ( 0,1.5) {$*$};

  \draw (minus) -- (star);
  \draw (plus) -- (star);
\end{tikzpicture}$$
on $\{ -,+,\ast\}$ and every upper interval is finite.  
Equivalently, $P$ is the face poset (excluding the empty face) of a locally finite complex $\Box_P$ of cubes, with a restriction analogous 
to that for $\Delta$-complexes in the simplicial case. We write $P^k$ for the elements of $P$ covering $2k$ elements and $P^K=\cup_{k\in K}P^k$ for the rank selected subposet of $P$.  

For each dimension $k\geq 1$ we construct a Garland poset $\gposet_kP
= (S,\leq,w,n,\rho)$ with 
$$S^{\{0,1\}}=P^{\{k,k+1\}} \text{ and } S^{-1}=\big\{\,[\,(a,b)\,]\,\big|\,P^{k-1}\cup\{\varphi\} \ni a< b \in P^k\,\big\}$$ with $[(a,b)]$ the vertices of 
the connected component containing $(a,b)$ of the graph $\Gamma$. 
Here $\Gamma$ is the graph with vertices $$\big\{\,(a,b)\in(\{\varphi\}\cup P^{k-1})\times P^k\,\big|\, a<b\hbox{ or } a=\varphi\,\big\}$$ and edges 
$$\big\{\,(\,b,c,b'\,)\in P^k\times P^{k+1}\times P^k\,\big|\,b<c>b'\not= b\,\big\}$$ where $(\,b,c,b'\,)$ connects the vertices $(b \bowtie b',b)$ and $(b \bowtie b',b')$ and finally $b \bowtie b'$ is the infimum of $b$ and $b'$ in $P^{\{k-1,k\}}$ if one exists and $\varphi$ otherwise.
We set $[(a,b)]<b$ 
and define $P^{\{k-1,k,k+1\}}_+=P^{\{k-1,k,k+1\}}\cup\{\varphi\}$ with 
$\varphi<b$ for any $b\in P^k$.  
By these definitions 
the map 
$\pi:S\rightarrow P^{\{k-1,k,k+1\}}_+$ with $\pi(\,[(a,b)]\,)=a$ and 
$\pi(c)=c$ for $c \in S^{\{0,1\}}$ is order preserving.  
Note that $\gamma:\Gamma\rightarrow P^{\{k,k+1\}}$ with $\gamma(p,b)=b$ and $\gamma(b,c,b')=c$ is an isomorphism between $\Gamma$ and $\Gamma_{\gposet_k P}$.  

If $n_0=2k+1$, $n_1=\binom{2k+2}{2}$ and $n_{010}=1$ and 
$\rho$ the obvious choice then $\gposet_k P$ clearly satisfies 
(P1), (P2), (P3) and (P4) in the definition of a Garland poset. 

For the cubical case the analog of a total ordering of the vertices used 
in the simplicial case is the following notion of $k$-collar. Here for 
$(\lambda_1,\ldots, \lambda_{k+1})
\in \bigwedge^{k+1}$ and $J = \{ j_1 < \cdots < j_r \}$ we write
$(\lambda_1,\ldots, \lambda_{k+1})_J$ for 
$(\lambda_{j_1},\ldots, \lambda_{j_r})$.

\begin{Definition}
A $k$-collar of $P$ is a choice of isomorphisms 
$\sigma_c:P_{\leq c}\rightarrow\bigwedge^{k+1}$ 
for each $c\in P^{k+1}$ such that they are compatible in the sense that if $c>b<c'$ 
there are two $k$-element subsets $J=[k+1]-\{j\}$ and $J'=[k+1]-\{j'\}$ 
such that the restrictions of $(\sigma_c)_J$ and $(\sigma_{c'})_{J'}$  
to $P_{\leq b}$ are the same isomorphism from $P_{\leq b}$ to $\bigwedge^ k$
and the restrictions of $(\sigma_c)_{\{j\}}$ and $(\sigma_{c'})_{\{j'\}}$ to $P_{\leq b}$ are the same constant (either $-$ or $+$ in $\bigwedge$). 
\end{Definition}

Not all cube complexes have such a choice. A complex $P$ is called 
$k$-monodromy free if it has a $k$-collar. In \ref{fig:mono} we show two 
typical scenarios which obstruct $k$-monodromy freeness for $k =1$. 
After fixing a $2$-cube $c$ and an isomorphism $\sigma_c$ 
to $\bigwedge^ 2$ in each example the images of the remaining vertices in 
$\bigwedge^2$ are determined.
Starting from a $2$-cube with blue labels and then 
greedily assigning the images of the neighboring $2$-cubes 
always yields a contradiction. 
In \ref{fig:mono} the red labels indicate the
vertices where obstructing monodromy freeness.

\begin{figure}[ht]
	\hskip-1cm\begin{subfigure}{0.3\textwidth}
		
		\begin{tikzpicture}[line join=round,line cap=round,scale=0.7]

\coordinate (A) at (-1,4);      
\coordinate (B) at (0,4);    
\coordinate (C) at (-1,3);    
\coordinate (D) at (0.4,2.6);  
\coordinate (E) at (4.3,1.4);  
\coordinate (F) at (2.3,1.7);  
\coordinate (G) at (-0.4,0);    
\coordinate (H) at (0.8,1.3);  

\fill[gray!30] (A) -- (B) -- (D) -- (C) -- cycle;      
\fill[gray!30] (B) -- (E) -- (F) -- (D) -- cycle;      
\fill[gray!30] (C) -- (D) -- (H) -- (G) -- cycle;             
\fill[gray!30] (H) -- (F) -- (E) -- (G) -- cycle;      

\draw[thick] (A) -- (B) -- (E) -- (G) -- (C) -- (A);

\draw[thick] (C) -- (D) -- (B);
\draw[thick] (D) -- (F) -- (E);
\draw[thick] (G) -- (H) -- (F);
\draw[thick] (D) -- (H) -- (G) ;
			\fill[red] (D) circle (2.5pt) node[right] {$--$};
			\fill[red] (C) circle (2.5pt) node[left] {$+-$};

			\fill[blue] (A) circle (2.5pt) node[left] {$++$};
			\fill[blue] (B) circle (2.5pt) node[above] {$-+$};
			\fill[blue] (E) circle (2.5pt) node[above] {$++$};
			\fill[blue] (F) circle (2.5pt) node[above] {$+-$};
			\fill[red] (G) circle (2.5pt) node[below] {$+-$};
			\fill[red] (H) circle (2.5pt) node[right] {$--$};

\end{tikzpicture}
		\caption{}
	\end{subfigure}
	\hskip3cm
	\begin{subfigure}{0.3\textwidth}
		\begin{tikzpicture}[scale=1, line join=round, line cap=round]

\def\a{2.0}
\def\b{0.5}

\fill[gray!30]
  plot[smooth, variable=\t, domain=180:360]
    ({(\a + \b*cos(\t/2))*cos(\t)},
    {0.42*(\a + \b*cos(\t/2))*sin(\t)})
  --
  plot[smooth, variable=\t, domain=360:180]
    ({(\a - \b*cos(\t/2))*cos(\t)},
     {0.42*(\a - \b*cos(\t/2))*sin(\t)})
  -- cycle;

\fill[gray!30]
  plot[smooth, variable=\t, domain=0:180]
    ({(\a + \b*cos(\t/2))*cos(\t)},
     {0.42*(\a + \b*cos(\t/2))*sin(\t)})
  --
  plot[smooth, variable=\t, domain=180:0]
    ({(\a - \b*cos(\t/2))*cos(\t)},
     {0.42*(\a - \b*cos(\t/2))*sin(\t)})
  -- cycle;

\draw[thick]
  plot[smooth, variable=\t, domain=0:360]
    ({(\a + \b*cos(\t/2))*cos(\t)},
     {0.42*(\a + \b*cos(\t/2))*sin(\t)});

\draw[thick]
  plot[smooth, variable=\t, domain=0:360]
    ({(\a - \b*cos(\t/2))*cos(\t)},
     {0.42*(\a - \b*cos(\t/2))*sin(\t)});

\foreach \t in {15,45,...,345} {
  \draw[black!60!black, thick]
    ({(\a - \b*cos(\t/2))*cos(\t)},
     {0.42*(\a - \b*cos(\t/2))*sin(\t)})
    --
    ({(\a + \b*cos(\t/2))*cos(\t)},
     {0.42*(\a + \b*cos(\t/2))*sin(\t)});
     \pgfmathparse{int(mod(int(((\t/15)-1)/2),2))< 1 ? int(1) : int(0)};
     \ifnum\pgfmathresult=0
         \fill[blue!60]
    ({(\a - \b*cos(\t/2))*cos(\t)},
     {0.42*(\a - \b*cos(\t/2))*sin(\t)})
			circle (1.5pt) node[right] {$++$};
		\else
			\ifnum \t > 15
  \fill[blue!60]
    ({(\a - \b*cos(\t/2))*cos(\t)},
     {0.42*(\a - \b*cos(\t/2))*sin(\t)})
     circle (1.5pt) node[right] {$+-$};
     \fi
     \fi
  \fill[blue!60]
    ({(\a + \b*cos(\t/2))*cos(\t)},
     {0.42*(\a + \b*cos(\t/2))*sin(\t)})
			circle (1.5pt);
}
  \fill[red!60]
    ({(\a + \b*cos(15/2))*cos(15)},
     {0.42*(\a + \b*cos(15/2))*sin(15)})
			circle (1.5pt) node[right] {$+-$};
         \fill[red!60]
    ({(\a - \b*cos(15/2))*cos(15)},
     {0.42*(\a - \b*cos(15/2))*sin(15)})
			circle (1.5pt) node[left] {$+-$};

\end{tikzpicture}
		\caption{}
	\end{subfigure}
	\caption{Two cubical complexes which are not monodromy free}
	\label{fig:mono}
\end{figure}


Here is another interpretation of $k$-monodromy freeness.
Call the connected component $[(a,b)]$ of $(a,b)$ in $\Gamma_{\gposet_k P}$ geometric if $a\in P^{k-1}$.
In this case it is a component of the link graph of $a$ in $\Box_P$ and transversal if 
$a=\varphi$ (see \ref{sec:6} for examples).  We call
$P$ transversally finite if all transversal components are finite
and in this case they will turn out to be another source of classes in 
$\widetilde{\Hom}^k(\Box_P)$ not arising in the simplicial case.  
Call $[(\varphi,b)]\in S^{-1}$ monodromy-free if the restriction of $\gamma$ to $\Gamma_{(\varphi,b)}$ extends to a cover preserving poset map to $P$.  This extension can be used to construct a collar and $P$ is $k$-monodromy free iff every element of $S^{-1}$ is (See \ref{sec:6} for examples).  

If $P$ is $k$-monodromy free then fix a compatible family 
$\{\sigma_c\}_{c\in P^{k+1}}$ as above and consider the induced orientation on 
$P^{\{ k-1,k,k+1\}}$ with $u_{ab}=(-1)^r$ if there is some $r$ with 
$(\sigma_b)_{\{r\}}$ restricted to $P_{\leq a}$ being the constant function 
$+\in\bigwedge$ and $u_{ab}=-(-1)^r$ otherwise.  (In the latter case there is a unique choice of $r$ with $(\sigma_b)_{\{r\}}$ restricted to $P_{\leq a}$ 
being the constant function $-\in\bigwedge$.)  Use $\pi$ to pull the 
orientation back to $S$ with $w_{xy}=u_{\pi(x)\pi(y)}$.  Now (P5) clearly holds and $(S,\leq,w,n,\rho)$ is a Garland poset.  

We adopt the notations for chain complexes and homology from the simplicial case. 
Note that $\pi$ again induces an injective map of three term cochain complexes $\pi^*:C^{k-1,k,k+1}\rightarrow B_{\gposet_kP}$ with cokernel 
$B^{-1}_{\gposet_kP} \slash C^{k-1}$ supported only in one degree.  
This time the cokernel is generated by both the image $L_{k,P}$ of the geometric link terms and the image $T_{k,P}$ of classes of transversal components of 
$\Gamma_{S}$ under the connecting homomorphism.
Thus by \ref{prop:garlandstructure} and \ref{thm:garland}:

\begin{cor}\label{cor:cubicalgarland} 
  If $P$ is a locally and transversally finite $k$-monodromy free cubicial 
  poset and the spectral gap 
  of every connected component of the link graph $\Gamma_{\gposet_kP}$ is 
	greater than $\frac{2k}{2k+1}$ then $\widetilde{\Hom}^k(\Box_P)=L_{k,P}+T_{k,P}$.  
\end{cor} 

If $P$ has $k$-monodromy there is a similar statement using the 
same Laplacian for geometric link graphs but one on a larger local system 
for the transversal links. 

The gap condition means that both the connected components of the geometric 
link graph and the transversal link graph have positive Laplacian spectra 
contained in $\{0\}\cup(\frac{2k}{2k+1},2]$.  
Note that if the link graph is regular of degree $d$ and has diameter at least four then the Alon-Boppona bound implies that this spectral bound can not be satisfied unless 
$d\geq 4k^2 $. 

\begin{Example}
  A $2$-dimensional torus gives somewhat misleading but easy to draw examples 
  of related simplicial and cubical complexes for which one can compare the 
  above criteria.  

  Write $\TT^ 2_\Box$ for the standard cubical subdivision of the $2$-dimensional 
  torus with $16$ vertices and $\TT^ 2_\Delta$ for its simplicial refinement also 
  with $16$ vertices and with all vertex degrees six.  The former is the 
  cellular product and the latter a simplicial product of two four cycles.  
  Write $P^\Box$ and $P^\Delta$ for their face posets. Take $k=1$ and note 
  that the associated Garland poset link graphs $\Gamma_{\gposet_1P^\Box}$ 
  and $\Gamma_{\gposet_1P^\Delta}$ are the disjoint union of $24$ squares 
  and $16$ hexagons respectively. The squares include $16$ arising as the 
  links of the $16$ vertices and $4$ transversal components in each of the 
  two directions.  The hexagons are just the links of the $16$ vertices.  

  Note, that the vertex links here are dual to cocycles representing classes 
  in $L_{1,P^\Box}$ and $L_{1,P^\Delta}$ which are both trivial but the 
  transversals are dual to nontrivial ones in $T_{1,P^\Box}$.  

  Write $\Lambda_G=I_{VG}-\frac{1}{deg}Adj_G$ the identity minus the adjacency 
  matrix divided by the vertex degrees for the normalized positive graph 
  Laplacian in the case that $G$ is regular and note that 
  $\Lambda_G\1={\bf 0}$ so $\{0\}\subseteq \Spec(\Lambda_G)\subseteq [0,2]$ and 
  in this case $\Lambda_{C_4}$ and $\Lambda_{C_6}$ are both tridiagonal 
  matrices with entries $0$, $1$ and $\frac{-1}{2}$ and 
  $\Spec(\Lambda_{C_4})=\Spec(\Lambda_{\Gamma_{\gposet_1P^\Box}})=\{0,1,2\}$ 
  while $\Spec(\Lambda_{C_6})=\Spec(\Lambda_{\Gamma_{\gposet_1P^\Delta}})=\{0, \frac{1}{2}, \frac{3}{2}, 2\}$, so the relevant spectral gaps are $1$ for the 
  cubical case and $\frac{1}{2}$ for the simplicial one.

  For the cubical case the hypotheses of \ref{cor:cubicalgarland} 
  hold since $1> 1-\frac{1}{2+1}=\frac{2}{3}$. Hence in this case the 
	conclusion also holds and $\widetilde{\Hom}^1(\TT^2_\Box)=\QQ^2=T_{1,P^\Box}$ which 
  is generated by various pairs of the eight transversal classes.

  For the simplicial case the hypotheses of \ref{cor:simplicialgarland} 
  fail since $\frac{1}{2}\not> 1-\frac{1}{1+1}=\frac{1}{2}$.  In this case 
	the conclusion also fails since $\widetilde{\Hom}^1(\TT^2_\Delta) =\QQ^2\neq\{0\}
  =L_{1,P^\Delta}$.
\end{Example}

\section{Proofs} \label{sec:proofs}
\label{sec:5}

Computations involving Garland structures arising from posets are clarified 
with a diagrammatic notation. Such a diagram is a $3$-level poset with nodes in 
levels from bottom to top corresponding to 
elements in $S^{-1}$, $S^0$ and $S^1$. 
Each node corresponds to a
summation over the elements of the corresponding $S^i$. 
Any type of edge between nodes 
indicates that the summations are coupled by an order relation. 
Nodes are maps $S^i \rightarrow \RR$. When denoted by $\ca$ for $i=-1$ or 
$\cb$ for $i=0$ they define $\ell_2$-cochains. 
We use ''$\bullet$'' to denote the 
constant map $1$ and a node marked with $a \in S^{-1}$, $b \in S^0$ or
$c \in S^1$ to denote the characteristic function of that element.
Cover relations indicated by double edges are weighted by $w$. 
At most one node per level can also have single or double parentheses.
This indicates that the summation index of this node is used as a subscript of 
the variable $z$, the full subscript is obtained by reading the subscripts
from bottom to top, duplicating doubly parenthesized indices.

\begin{Example}
The formula
$$\sum_{c\in S^1}\sum_{b\in S^0_{<c}}\sum_{a\in S^{-1}_{<b}}z_c \,w_{bc} \,\cb(b)\,=\,\sum_{c\in S^1}\sum_{b\in S^0_{<c}}\sum_{a\in S^{-1}_{<b}}\sum_{a'\in S^{-1}_{<c}}z_{a'cc}\, w_{bc}\, \cb(b)$$ in $B^1$ becomes
\end{Example}


\begin{center}
\vskip5pt
\begin{tikzpicture}
  \node (top1) at (-1,0) {$(\bullet)$};
  \node [below of=top1] (mid1) {$\cb$};
  \node [below of=mid1] (bot1) {$\bullet$};
	\draw [black,  thick, shorten <=-2pt, shorten >=-2pt] (mid1) -- (bot1);
    \draw [double, black,  thick, shorten <=-2pt, shorten >=-2pt] (top1) -- (mid1);

    \node [right =.4cm of mid1] (eq3) {\begin{Large} =\end{Large}};

    \node (top3) at (2,0) {$(\!(\bullet)\!)$};
  \node [below left=0.4cm and 0.3cm of top3] (mid3) {$\cb$};
  \node [below of=mid3] (bot13) {$\bullet$};
  \node [below=1.4cm of top3] (bot23) {$(\bullet)$};
   \draw [black,  thick, shorten <=-2pt, shorten >=-2pt] (mid3) -- (bot13);
    \draw [double, black,  thick, shorten <=-2pt, shorten >=-2pt] (top3) -- (mid3);
	\draw [black,  thick, shorten <=-2pt, shorten >=-2pt] (top3) -- (bot23);
\end{tikzpicture}
\end{center}

  Note that for every pair of disjoint diagrams $\star$ and $*$ the
  identity 
  \vskip1pt
	\begin{center}
		\begin{align}
\label{id:0}
  \begin{aligned}
\begin{tikzpicture}
  \node (top3) at (3,0) {$\star$};
  \node [below of=top3] (mid3) {$\bullet$};
  \node [below of=mid3] (bot3) {$*$};
\draw [double, black,  thick, shorten <=-2pt, shorten >=-2pt] (mid3) -- (bot3);
    \draw [double, black,  thick, shorten <=-2pt, shorten >=-2pt] (top3) -- (mid3);
\node [right =.5cm of mid3] (eq1) {\begin{Large} =\end{Large}};
\node (top4) at (6,0) {};
\node [below of=top4] (mid4) {$\mathlarger{\mathlarger{0}}$};
\end{tikzpicture}
  \end{aligned}
		\end{align}
	\end{center}

  holds by the orientation condition \eqref{eq:w} of (P5). 

\medskip

\begin{proof}[Proof of \ref{lem:induced}]

  \noindent{\sf $A_\gposet$ is a sequence of bounded maps:} 
	Since the $z_{ax}$ form orthogonal bases of the respective spaces 
	for boundedness of the maps in $A_\gposet$ it suffices to
	show that $\fa_i(z_{az})$ is uniformly bounded by a constant times
	$\|z_{az}\|$. 

  If $z_{aa}$ is a basis element of $A^{-1}$ then by definition 
  $\|z_{aa}\|^2$ is the number of $c \in S^1$ with $a < c$ and
	\begin{align*}
		\| \fa_0(z_{aa})\|^2 & = \|\sum_{ b \in S_{>a}^0} w_{ab} z_{ab} \|^2
		\\ & = \| \sum_{ b \in S_{>a}^0} w_{ab} \sum_{ c \in S_{>b}^1} z_{abc}\|^2
		\\ & = \sum_{ c \in S_{>a}^1} \sum_{b \in S_{>a,<c}^0} \| z_{abc}\|^2
		\overset{(P5)}{\leq} 2 \| z_{aa}\|^2. 
	\end{align*}
  If $z_{ab}$ is a basis element of $A^{0}$ then 
	\begin{align*}
		\| \fa_1(z_{ab})\|^2 & = \|\sum_{c \in S_{>b}^1} w_{bc} z_{ac} \|^2
		\\ & = \sum_{c \in S_{>b}^1} \| w_{bc} z_{acc} \|^2 = |S_{> b}^1| = \|z_{ab}\|^2.
	\end{align*}

Since
	\begin{align*}
		\fa_1(\fa_0(z_{aa})) & = \fa_1(\sum_{b \in S_{>a}^0} w_{ab} z_{ab}) \\ & = 
		\sum_{ b \in S_{>a}^0} w_{ab} \sum_{c\in S_{>b}^1} w_{bc} z_{ac}) \\ & = 
	\sum_{c \in S_{>a}^1} z_{ac} \sum_{b \in S_{>a,<c}^0} w_{ab} w_{bc}  
	\overset{(P5)}{=} 0
	\end{align*}
{\sf $A_\gposet~:~A^ {-1} \xrightarrow{\fa_0} A^0 \xrightarrow{\fa_1} A^1$ is a complex}. 

	\smallskip

  \noindent{\sf $B_\gposet$ is a sequence of bounded maps} 
 since both $\fb^i$ are a composition of an orthogonal projection
to a closed subspace and a bounded map.

\smallskip
To show that 
\noindent{\sf $B_\gposet$ is a complex} 
we first need to consider the two projection maps $\fp_{B^i}$.
	For $\fp_{B^0}$ we need to check that $\fp_{B^ 0}(z_{ab}) = \frac{1}{n_0}z_{b}$ for $a  \in S^{-1}$ and $b \in S^ 0$. For that note that by 
	definition 
	$\frac{(z_{ab},z_{b'})}{(z_{b'},z_{b'})}= 0$ if $b \neq b'$ and 
	$\frac{(z_{ab},z_{b'})}{(z_{b'},z_{b'})}= \frac{1}{n_0}$ if $b = b'$.
 The latter case follows from the below two calculations together with (P1).  
 Since this is the first time we use a diagram proof we add the actual calculations in this and in the next case.

\begin{center}
    \begin{tikzpicture}
	  \node (top1) at (-10,0) {$\bullet$};
	  \node [below of=top1] (mid1) {$b$};
  \node [below of=mid1] (bot1) {$a$};
\draw [black,  thick, shorten <=-2pt, shorten >=-2pt] (mid1) -- (bot1);
\draw [black,  thick, shorten <=-2pt, shorten >=-2pt] (top1) -- (mid1);
	    \node [left =0cm of mid1] (co1) {\begin{Large} $(z_{ab},z_{b})$ = \end{Large} };
		  \node (top2) at (-6,0) {$\bullet$};
  \node [below of=top2] (mid2) {$b$};
  \node [below of=mid2] (bot2) {$\bullet$};
	  \node [left =-0.2cm of mid2] (co1) {\begin{Large} $(z_b,z_b)$ = \end{Large} };
\draw [black,  thick, shorten <=-2pt, shorten >=-2pt] (mid2) -- (bot2);
\draw [black,  thick, shorten <=-2pt, shorten >=-2pt] (top2) -- (mid2);
\end{tikzpicture}

\end{center}

Alternatively, 

\begin{align*}
	\fp_{B^ 0} (z_{ab}) & = \sum_{b' \in S^0} \frac{(z_{ab},z_{b'})}{(z_{b'},z_{b'})} z_{b'} 
= 
\frac{(z_{ab},z_{b})}{(z_{b},z_{b})} z_{b}
\end{align*}
and
\begin{align*} 
	(z_b,z_b)& = 
	(\sum_{a \in S^{-1}_{< b}} z_{ab}, \sum_{a \in S^{-1}_{< b}}z_{ab})
	= \sum_{a \in S^{-1}_{< b}} (z_{ab},z_{ab}) \\
	& =  \sum_{c' \in S^ {1}_{> b}}\sum_{c \in S^1_{> b}} \sum_{a \in S^{-1}_{< b}} (z_{abc},z_{abc'}) \\ 
	& = \sum_{c \in S^ {1}_{> b}} \sum_{a \in S^{-1}_{< b}} (z_{abc},z_{abc}) = \big|\, S^{1}_{> b} \times S_{< b}^{-1}\,\big|  
= \big|\, S^{1}_{> b}\,\big| \cdot n_0 
\end{align*}
and
\begin{align*} 
	(z_{ab},z_b)& = 
	(z_{ab} , \sum_{a' \in S^{-1}_{< b}} z_{a'b}) 
	 = 
	(z_{ab} , z_{ab})= 
	\sum_{c \in S^{1}_{> b}}
	(z_{abc} , z_{abc})= \big|\,S_{> b}^1 \,\big|
\end{align*}
yield the result.

For $\fp_{B^1}$ we need to check that 
	$\fp_{B^ 1}(z_{ac}) = \frac{1}{n_1}z_{c}$ for $a  \in S^{-1}$ and $c \in S^ 1$. For that note that by 
	definition
	$\frac{(z_{ac'},z_{c})}{(z_c,z_c)} = 0$ if $c \neq c'$ and
	$\frac{(z_{ac'},z_{c})}{(z_c,z_c)} = \frac{1}{n_1}$ 
	if $c = c'$.
   Again in the latter case 
	result follows from the below two calculations
together with (P2) and (P5). 

\begin{center}
    \begin{tikzpicture}
	  \node (top1) at (-10,0) {$c$};
	  \node [below of=top1] (mid1) {};
  \node [below of=mid1] (bot1) {$a$};
\draw [black,  thick, shorten <=-2pt, shorten >=-2pt] (top1) -- (bot1);
\draw [black,  thick, shorten <=-2pt, shorten >=-2pt] (top1) -- (mid1);
	  \node [left =0cm of mid1] (co1) {\begin{Large} $(z_{ac},z_c)$ = \end{Large} };
		  \node (top2) at (-6,0) {$c$};
  \node [below of=top2] (mid2) {};
  \node [below of=mid2] (bot2) {$\bullet$};
	  \node [left =-0.2cm of mid2] (co1) {\begin{Large} $(z_c,z_c)$ = \end{Large} };
\draw [black,  thick, shorten <=-2pt, shorten >=-2pt] (top2) -- (bot2);
\draw [black,  thick, shorten <=-2pt, shorten >=-2pt] (top2) -- (mid2);
\end{tikzpicture}

\end{center}

Again alternatively, 
\begin{align*}
	\fp_{B^ 0} (z_{ac}) & = \sum_{c' \in S^0} \frac{(z_{ac},z_{c'})}{(z_{c'},z_{c'})} z_{c'} 
= 
\frac{(z_{ac},z_{c})}{(z_{c},z_{c})} z_{c}
\end{align*}

and

\begin{align*}
	(z_c,z_c)& = 
	\sum_{a \in S^{-1}_{< c}}
	(z_{ac},z_{ac})& = 
	\sum_{a \in S^{-1}_{< c}}
	\sum_{b \in S^{0}_{< c} \cap S^ 0_{> a}}
	(z_{abc},z_{abc})
	= 2 \cdot n_1
\end{align*}

\begin{align*} 
	(z_{ac},z_c)& = 
	(z_{ac},z_{ac}) = 
	\sum_{b \in S^{0}_{< c} \cap S^ 0_{> a}}
	(z_{abc},z_{abc})
	=2.
\end{align*}

Now we can show that $B_\gposet$ is a complex:\vskip1pt

\begin{center}
  \begin{tikzpicture}
    \node (top1) at (-3,0) {};
    \node [below of=top1] (mid1) {};
  \node [below of=mid1] (bot1) {($\ca$)};
	  \node [left =0.1cm of mid1] (co1) {$\mathlarger{\mathlarger{\fb_1\,\fb_0}}$ \text{\begin{Huge} ( \end{Huge}}};
		  \node [right =-.1cm of mid1] (eq1) {\text{\begin{Huge}\,\,)\end{Huge} \begin{Large} =\end{Large}}};
  \node (top2) at (2.0,0) {};
  \node [below of=top2] (mid2) {($\bullet$)};
  \node [below of=mid2] (bot2) {($\ca$)};
\draw [double, black,  thick, shorten <=-2pt, shorten >=-2pt] (mid2) -- (bot2);
	  \node [left = -0.0cm of mid2] (co2) {$\mathlarger{\mathlarger{\fp_{B^1}\,\fa_1\,\fp_{B^0}}}$ \hskip-0.2cm\text{\begin{Huge} (\end{Huge}}};
		  \node [right =-.2cm of mid2] (eq2) {\text{\begin{Huge} ) \end{Huge} \begin{Large} = \end{Large}}};
\end{tikzpicture}
\end{center}

\smallskip

\begin{center}
\begin{tikzpicture}
  \node (top3) at (0,0) {};
  \node [below of=top3] (mid3) {($\bullet$)};
  \node [below of=mid3] (bot13) {$\ca$};
  \node [below left=0.3cm and 0.2cm of mid3] (bot23) {($\bullet$)};
\draw [double, black,  thick, shorten <=-2pt, shorten >=-2pt] (mid3) -- (bot13);
\draw [black,  thick, shorten <=-2pt, shorten >=-2pt] (mid3) -- (bot23);
	\node [left =0.01cm of mid3] (co3) {$\mathlarger{\mathlarger{\frac{1}{n_0}\, \fp_{B^1}\, \fa_1 \, \text{\begin{Huge} ( \end{Huge}}}}$};
		\node [right =-.2cm of mid3] (eq3) {\text{\begin{Huge})\end{Huge} \hskip-0.1cm\begin{Large} =\end{Large}} };
  \node (top1) at (4.3,0) {($\bullet$)};
  \node [below of=top1] (mid1) {$\bullet$};
  \node [below of=mid1] (bot11) {$\ca$};
  \node [below left=0.4cm and 0.3cm of mid1] (bot21) {($\bullet$)};
\draw [double, black,  thick, shorten <=-2pt, shorten >=-2pt] (mid1) -- (bot11);
\draw [black,  thick, shorten <=-2pt, shorten >=-2] (mid1) -- (bot21);
\draw [double, black,  thick, shorten <=-2pt, shorten >=-2pt] (top1) -- (mid1);
	\node [left =0.36cm of mid1] (co1) {$\mathlarger{\mathlarger{\frac{1}{n_0}\,\fp_{B^1}}}$\text{\begin{Huge} ( \end{Huge}}};
		\node [right =.4cm of mid1] (eq1) {\text{\begin{Huge})\end{Huge} \begin{Large} =\end{Large}}};
\end{tikzpicture}
\end{center}

\smallskip

\begin{center}
\begin{tikzpicture}
  \node (top2) at (0,0) {$(\bullet)$};
  \node [below left=0.5cm and 0.3cm of top2] (mid2) {$\bullet$};
  \node [below left=0.55cm and 0.3cm of mid2] (bot12) {$\bullet$};
  \node [below of=mid2] (bot22) {$\ca$};
  \node [below=1.4cm of top2] (bot32) {$(\bullet)$};
	\node [left=.5cm of mid2] (co2) {$\mathlarger{\mathlarger{\frac{1}{n_1\,n_0}}}$};
   \draw [black,  thick, shorten <=-2pt, shorten >=-2pt] (mid2) -- (bot12);
   \draw [double, black,  thick, shorten <=-2pt, shorten >=-2pt] (mid2) -- (bot22);
    \draw [double, black,  thick, shorten <=-2pt, shorten >=-2pt] (top2) -- (mid2);
	\draw [black,  thick, shorten <=-2pt, shorten >=-2pt] (top2) -- (bot32);
  \node [right =1cm of mid2] (eq2) {\begin{Large} =\end{Large}};

  \node (top3) at (3.5,0) {$(\bullet)$};
  \node [below left=0.4cm and 0.3cm of top3] (mid3) {$\bullet$};
  \node [below of=mid3] (bot13) {$\ca$};
  \node [below=1.3cm of top3] (bot23) {$(\bullet)$};

  \node [left=-.1cm of mid3] (co3) {$\mathlarger{\mathlarger{\frac{1}{n_1}}}$};
   \draw [double, black,  thick, shorten <=-2pt, shorten >=-2pt] (mid3) -- (bot13);
    \draw [double, black,  thick, shorten <=-2pt, shorten >=-2pt] (top3) -- (mid3);
	\draw [black,  thick, shorten <=-2pt, shorten >=-2pt] (top3) -- (bot23);

\node [right =.8cm of mid3] (eq3) {\begin{Large} = \end{Large}};
\node (top4) at (5,0) {};
\node [below of=top4] (mid4) {$\mathlarger{\mathlarger{0}}$};
\end{tikzpicture}
\end{center}
    
  The first and third equalities follow immediately from the definitions, the
  second and fourth equalities use the calculations of the projection maps.
  The fifth equality uses (P1) and the last equality follows from \ref{id:0}.
  
  \medskip
\end{proof}

\medskip

\begin{proof}[Proof of \ref{lem:asum}]
  By \ref{lem:induced} we know that $A_\gposet$ is a complex.
  Since each $A_a$ is closed under the two maps $\fa_i$ it follows that these are subcomplexes.
  For the exactness of each $A_a$ note that $A^{-1}_a\ni z_{aa}$ is one-dimensional with 
  $w_{ab}$ the coefficient of $z_{ab}$ in $\fa_0(z_{aa})$. By (P4) $S_{>a}$ is connected. If
  $S^{1}_{>a} = \emptyset$ then $S_{> a} = \{b\}$ or $ \emptyset$ and we are done. For $c \in S_{>a}^1$
  we have by (P5) that $S_{>a} \cap S_{< c} = \{b,b'\}$ with $b \neq b'$. 
  Assume there is $\phi\in \ker(\fa_1) \cap A_a^0$ which is not a boundary. We may choose 
  $\phi$ such that the number of $b \in S^0_{>a}$ in its support is minimal for cycles which are not boundaries. 
  For the coefficients $v_{ab}$, $v_{ab'}$ of
  $z_{ab}$ and $z_{ab'}$ in $\phi$ we then have $0=w_{bc}v_{ab} +w_{b'c}v_{ab'}$. 
  By (P5) this implies that $v_{ab} = w_{ab} v'$ and $v_{ab'} = w_{ab'} v'$ for some 
  $v'$. It follows that $\phi - \fa_0(v'\,z_{aa})$ is 
  a cycle supported on fewer terms. This shows that each $A_a$ is exact. 

  By the construction of $A_a^i$ the given orthogonal basis
  for $A^i$ is the disjoint union over $a$ of the given
  orthogonal bases of the $A_a^i$.
  Hence for each $i$ the sum over $a$ of the $A_a^i$  is an orthogonal sum which
  contains a basis of $A^i$. 
  It follows that $A_\gposet$ is the closure of the sum over $a$ of the exact complexes $A_a$ and
  hence $A_\gposet$ is an exact complex. 
\end{proof}

\medskip

\begin{proof}[Proof of \ref{prop:garlandstructure}] ~\\ 

  By \ref{lem:induced} and \ref{lem:asum} we have that $G$ is a Garland structure.

  \smallskip

\noindent{\sf (i):} 
  Take $\phi\in A^0$ of norm one and write $\phi = \sum_{a \in S'} \phi_a \in A^0$ with $0 \neq \phi_a \in A_a^0$ and $S' \subseteq S^{-1}$.
  The $\phi_a$ and their images under $\fa_1$ are mutually 
  orthogonal and 
  \begin{align*} 
	  \| \fa_1(\phi)\|^2 & = \sum_{a \in S'} \| \fa_1(\phi_a)\|^2 \\
	   & = \sum_{a \in S'} \| \phi_a \|^2 \cdot \| \fa_1(\frac{1}{\|\phi_a\|} \,\phi_a)\|^2\\
	   & \geq \sum_{a \in S'} \|\phi_a\|^2 \cdot \alpha_{A_a} \\ 
	  & \geq \inf_{a \in S^{-1}} \alpha_{A_a} \cdot 1.
  \end{align*}

  Hence $\alpha_{A_{S}} \geq \inf_{a \in S^{-1}} \alpha_{A_a}$.
  The other inequality is immediate.

  \medskip

\noindent{\sf (ii):} 
  Assume $\ker \fb_1^H = \ker \fb_1 \cap B_H^0 \neq 0$. 
  We want to evaluate the right hand side of \eqref{eq:beta} for 
  $\kappa=\frac{n_{010}n_1}{n_0^2}$.
  First, we compute the two terms in the numerator of the right hand side of \eqref{eq:beta}.
  Consider $\phi = \sum_{b \in S^{0}} \cb(b) z_b \in B^0$. Note that in our notation
  $\phi = ( \cb)$.
  Compute:

\begin{center}
  \begin{tikzpicture}
    \node (top1) at (0,0) {};
    \node [below of=top1] (mid1) {($\cb$)};
    \node [left =0cm of mid1] (nm11) {$\Big\|$};
    \node [right =0cm of mid1] (nm12) {$\Big\|^2$};

    \node [right =.5cm of mid1] (eq1) {\begin{Large} =\end{Large}};

    \node (top3) at (3,0) {($\bullet$)};
    \node [below of=top3] (mid3) {($\cb$)};
    \node [below of=mid3] (bot3) {($\bullet$)};
    \node [left =0cm of mid3] (nm31) {$\Big\|$};
    \node [right =0cm of mid3] (nm32) {$\Big\|^2$};
    \draw [black,  thick, shorten <=-2pt, shorten >=-2pt] (mid3) -- (bot3);
    \draw [black,  thick, shorten <=-2pt, shorten >=-2pt] (top3) -- (mid3);

    \node [right =.5cm of mid3] (eq3) {\begin{Large} =\end{Large}};

    \node (top4) at (6,0) {$\bullet$};
    \node [below of=top4] (mid4) {$\cb^2$};
    \node [below of=mid4] (bot4) {$\bullet$};
    \draw [black,  thick, shorten <=-2pt, shorten >=-2pt] (mid4) -- (bot4);
    \draw [black,  thick, shorten <=-2pt, shorten >=-2pt] (top4) -- (mid4);

    \node [right =0.1cm of mid4] (eq4) {\begin{Large} =\end{Large}};

    \node (top5) at (9,0) {$\bullet$};
    \node [below of=top5] (mid5) {$\cb^2$};
    \node [left =0.1cm of mid5] (co) {$\mathlarger{\mathlarger{n_0}}$};
    \draw [black,  thick, shorten <=-2pt, shorten >=-2pt] (top5) -- (mid5);
  \end{tikzpicture}
\end{center}

  \vskip20pt

\begin{center}
  \begin{tikzpicture}
    \node (top1) at (-4,0) {};
    \node [below of=top1] (mid1) {$\mathlarger{\mathlarger{\fa_1}}\text{\begin{Huge}\Huge(\end{Huge}}\,(\cb)\,\text{\begin{Huge}\Huge)\end{Huge}}$};
    \node [left =0cm of mid1] (nm11) {$\Big\|$};
    \node [right =0cm of mid1] (nm12) {$\Big\|^2$};

    \node [right =.5cm of mid1] (eq1) {\begin{Large} =\end{Large}};

    \node (top3) at (-1,0) {$(\!(\bullet)\!)$};
    \node [below of=top3] (mid3) {$\cb$};
    \node [below of=mid3] (bot3) {$(\bullet)$};
    \node [left =0cm of mid3] (nm31) {$\Big\|$};
    \node [right =0cm of mid3] (nm32) {$\Big\|^2$};
    \draw [black,  thick, shorten <=-2pt, shorten >=-2pt] (mid3) -- (bot3);
    \draw [double, black,  thick, shorten <=-2pt, shorten >=-2pt] (top3) -- (mid3);

    \node [right =.5cm of mid3] (eq3) {\begin{Large} =\end{Large}};

    \node (top4) at (2,0) {$\bullet$};
    \node [below left=0.5cm and 0.3cm of top4] (mid14) {$\cb$};
    \node [below right=0.5cm and 0.3cm of top4] (mid24) {$\cb$};
    \node [below right=0.5cm and 0.3cm of mid14] (bot4) {$\bullet$};
    \draw [black,  thick, shorten <=-2pt, shorten >=-2pt] (mid14) -- (bot4);
    \draw [black,  thick, shorten <=-2pt, shorten >=-2pt] (mid24) -- (bot4);
    \draw [double, black,  thick, shorten <=-2pt, shorten >=-2pt] (mid14) -- (top4);
    \draw [double, black,  thick, shorten <=-2pt, shorten >=-2pt] (mid24) -- (top4);

    \node [right =0.1cm of mid24] (eq4) {\begin{Large} $\overset{(P3)}{=}$\end{Large}};
  \end{tikzpicture}
  
  \begin{tikzpicture}
    \node (top5) at (0,0) {$\bullet$};
    \node [below left=0.5cm and 0.3cm of top5] (mid15) {$\cb$};
    \node [below right=0.5cm and 0.3cm of top5] (mid25) {$\cb$};
    \node [left =0.1cm of mid15] (co) {$\mathlarger{\mathlarger{n_{010}}}$};
    \draw [double, black,  thick, shorten <=-2pt, shorten >=-2pt] (mid15) -- (top5);
    \draw [double, black,  thick, shorten <=-2pt, shorten >=-2pt] (mid25) -- (top5);

    \node [right =0.1cm of mid25] (pl) {\begin{Large} +\end{Large}};

    \node (top6) at (6,0) {$\bullet$};
    \node [below of=top6] (mid6) {$\cb^2$};
    \node [left =0.1cm of mid6] (co) {$\mathlarger{\Big( \mathlarger{n_0-n_{010}\Big)}}$};
    \draw [black,  thick, shorten <=-2pt, shorten >=-2pt] (mid6) -- (top6);
  \end{tikzpicture}
\end{center} 
\vskip20pt  
\begin{center}
  \begin{tikzpicture}
    \node (top1) at (-4,0) {};
    \node [below of=top1] (mid1) {$\mathlarger{\mathlarger{\fb_1}}\,\text{\begin{Huge} ( \end{Huge}}\,(\cb)\,\text{\begin{Huge})\end{Huge}}$};

    \node [left =-0.2cm of mid1] (nm11) {$\Big\|$};
    \node [right =-0.2cm of mid1] (nm12) {$\Big\|^2$};

    \node [right =.2cm of mid1] (eq1) {\begin{Large} =\end{Large}};

    \node (top3) at (1,0) {$(\!(\bullet)\!)$};
    \node [below left=0.4cm and 0.3cm of top3] (mid3) {$\cb$};
    \node [below of=mid3] (bot13) {$\bullet$};
    \node [below=1.4cm of top3] (bot23) {$(\bullet)$};
    \node [right =1.2cm of mid3] (nm32) {$\Big\|^2$};
    \node [left=-.1cm of mid3] (co3) {$\mathlarger{\mathlarger{\frac{1}{n_1}}}$};
    \node [left=-.1cm of co3] (nm31) {$\Big\|$};
    \draw [black,  thick, shorten <=-2pt, shorten >=-2pt] (mid3) -- (bot13);
    \draw [double, black,  thick, shorten <=-2pt, shorten >=-2pt] (top3) -- (mid3);
    \draw [black,  thick, shorten <=-2pt, shorten >=-2pt] (top3) -- (bot23);

    \node [right =1.5cm of mid3] (eq3) {\begin{Large} =\end{Large}};

    \node (top4) at (5,0) {$\bullet$};
    \node [below left=0.5cm and 0.3cm of top4] (mid14) {$\cb$};
    \node [below right=0.5cm and 0.3cm of top4] (mid24) {$\cb$};
    \node [below of=mid14] (bot14) {$\bullet$};
    \node [below of=mid24] (bot34) {$\bullet$};
    \node [below=1.6cm of top4] (bot24) {$\bullet$};
    \node [left =0.1cm of mid14] (co) {$\mathlarger{\mathlarger{\frac{1}{n_1^2}}}$};
    \draw [double, black,  thick, shorten <=-2pt, shorten >=-2pt] (top4) -- (mid14);
    \draw [double, black,  thick, shorten <=-2pt, shorten >=-2pt] (top4) -- (mid24);
    \draw [black,  thick, shorten <=-2pt, shorten >=-2pt] (top4) -- (bot24);
    \draw [black,  thick, shorten <=-2pt, shorten >=-2pt] (mid24) -- (bot34);
    \draw [black,  thick, shorten <=-2pt, shorten >=-2pt] (mid14) -- (bot14);
  \end{tikzpicture}
\end{center}
  
\begin{center}
  \begin{tikzpicture}
    \node (top5) at (7,0) {$\bullet$};
    \node [below left=0.5cm and 0.3cm of top5] (mid15) {$\cb$};
    \node [below right=0.5cm and 0.3cm of top5] (mid25) {$\cb$};
    \node [left =0.1cm of mid15] (co) {$\mathlarger{\mathlarger{\frac{n_0^2}{n_1}}}$};
    \node [left =1.8cm of mid15] (eq5) {\text{\begin{Huge} =\end{Huge}}};
    \draw [double, black,  thick, shorten <=-2pt, shorten >=-2pt] (top5) -- (mid15);
    \draw [double, black,  thick, shorten <=-2pt, shorten >=-2pt] (top5) -- (mid25);
  \end{tikzpicture}
\end{center}
\vskip5pt
Thus the right hand side of \eqref{eq:beta} 
evaluates to 
$$\frac{\|\fa_1(\cb)\|^2 - \frac{n_{010}n_1}{n_0^2}\|\fb_1(\cb)\|^2}{\|\cb\|^2}=\frac{n_0-n_{010}}{n_0}.$$  
\end{proof}

\section{Examples}
\label{sec:6}

In this section we discuss how examples of simplicial and cubical complexes 
can be constructed to which \ref{cor:cubicalgarland} applies. 

\subsection{Moment angle complexes and branched coverings}
For every simplicial complex there is an associated moment angle cubical 
complex as described below (see also Chapter 4 of \cite{BP} for more 
details ). The Garland posets associated to these have 
geometric link graphs which are $1$-skeleta of links of the initial simplicial 
complex. The transversal link graphs will be $1$-skeleta of cubes and hence 
have small spectral gap. Taking branched covers will yield link graphs 
which are covers of the original and do not increase the spectral gaps.  
Finally a different quotient for which the constructed cover is again
a branched cover will produce link graphs 
which are quotients of a cover of the original. If the original link 
graph is regular this yields a fairly arbitrary regular link graph. 


If $K\subseteq 2^V$ is a simplicial complex then the associated moment 
angle subcomplex $X_K$ of $[0,1]^V$ is
$$X_K=\bigcup_{\genfrac{}{}{0pt}{}{\sigma \in K}{\omega, \nu}} [0,1]^\sigma\times\{0\}^\omega\times\{1\}^\nu$$

\medskip
\noindent where $\sigma$, $\omega$ and $\nu$ partition $V$.
For the simplicial complex $K$ on $V = \{1,2,3\}$ and with maximal faces
$\{1,2\}$ and $\{1,3\}$ 
\ref{fig:first}(B) depicts the corresponding moment angle complex
$X_K$ 
with the empty squares corresponding to directions '2" and "3". 
Note that since a moment angle complex $X_K$ is a 
subcomplex of a cube for each $k$ its level $k$ Garland poset $S$ is 
$k$-monodromy-free by the second criterion.  It has link graph

$$\Gamma_{S}=
\underbrace{\mathop{\coprod}_{\genfrac{}{}{0pt}{}{\sigma\in K^{k-2}}{\omega\subseteq V \setminus \sigma}} \link^1_K(\sigma)}_{\text{geometric link graph}} \amalg  
\underbrace{\mathop{\coprod}_{\genfrac{}{}{0pt}{}{\tau\in K^{k-1}}{\omega\subseteq V \setminus \tau}} Q_{|\link^0_K(\tau)|}}_{\text{transversal link graph}}$$

\noindent where $K^i$ is the set of $i$-faces of $K$, 
$\link^i_K(\sigma)$ is the $i$-skeleton of the simplicial link of $\sigma$ in 
$K$ and $Q_i$ is the $1$-skeleton of an $i$-dimensional cube. 

Note that for $r \geq 1$ the spectral gap for $Q_r$ is only $\frac{2}{r}$ so 
while the spectral gaps for the geometric link components may be large, 
those of the transversal ones are not. Indeed, the hypotheses of 
\ref{cor:cubicalgarland} for the 
transversal graph can only hold if $r \leq 2+\frac{1}{k}$. Hence the 
$k$-skeleton of $K$ is a pseudomanifold with boundary so the
geometric link graphs are paths or cycles. A computation shows that for the spectral gap to be more than $\frac{2k}{2k+1}$ there are 
at most three vertices in the components of the geometric link graph which are
paths and at most four in the components which 
are cycles (unless $k=1$ where a pentagon is also allowed). Hence \ref{cor:cubicalgarland} applies to the torus which is the moment angle complex of the 
boundary of a cross-polytope but very few others.

Next we discuss if and how branched covers can change the situation. We
start with a purely poset theoretic definition of branched covers. 

\begin{Definition} 
  A map $f:P\rightarrow Q$ of cubical posets is a $k$-branched cover 
  if for any $q\in Q$ with rank $k$ every restriction of 
  $f: f^{-1}\big(\,Q_{\geq q}\,\big)\rightarrow Q_{\geq q}$ 
  to a connected component of the domain is an isomorphism.
  If the number of connected components is the same for every choice of $q$ 
  call this number the degree of $f$.  
\end{Definition}

Geometrically this corresponds to the following.  
If $X$ is a cubical complex and $Y^\circ$ is a topological cover of the subspace $X^\circ$ of $X$ obtained by deleting all cubes of dimension less than $k$ then the associated cubical complex $Y$ is $k$-branched over $X$.  The covering map followed by the inclusion gives a map $\pr^\circ$ from $Y^\circ$ to $X$.  A universal description of $Y$ is obtained by factoring $\pr^\circ$ into a dense inclusion into $Y$ followed by a proper map $\pr$ to $X$.  The cubical structure on $Y$ is obtained from the decompositions of $Y^\circ$ and $X$ into open cubes.  A synthetic description of $Y$ is obtained by expressing $Y^\circ$ as a gluing (colimit) of spaces $C^\circ$ which are cubes $C$ (of dimension at least $k$) with all faces of dimension less than $k$ deleted.  The complex $Y$ has the same gluing diagram but with each deleted cube $C^\circ$ replaced by the full cube $C$.  

Note that if $Q$ and $P$ are the face posets of $Y$ and $X$ respectively 
then the graph $\Gamma_{\gposet_kQ}$ is a cover of the graph 
$\Gamma_{\gposet_kP}$ with the map induced by $\pr$.  If $Y^\circ$ is the universal cover of $X^\circ$ the domain is a forest (see \ref{fig:second}).  

\medskip

\begin{figure}[h]
\begin{picture}(0,0)%
\includegraphics{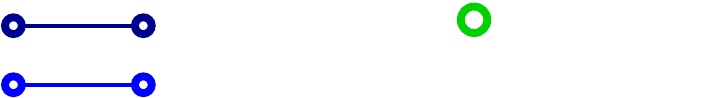}%
\end{picture}%
\setlength{\unitlength}{4144sp}%
\begingroup\makeatletter\ifx\SetFigFont\undefined%
\gdef\SetFigFont#1#2#3#4#5{%
  \reset@font\fontsize{#1}{#2pt}%
  \fontfamily{#3}\fontseries{#4}\fontshape{#5}%
  \selectfont}%
\fi\endgroup%
\begin{picture}(5416,735)(1654,-1062)
\put(3151,-511){\makebox(0,0)[lb]{\smash{{\SetFigFont{12}{14.4}{\rmdefault}{\mddefault}{\updefault}{\color[rgb]{0,0,0}transversal link graph}%
}}}}
\put(3151,-961){\makebox(0,0)[lb]{\smash{{\SetFigFont{12}{14.4}{\rmdefault}{\mddefault}{\updefault}{\color[rgb]{0,0,0}geometric link graph of $0$-cubes}%
}}}}
\put(5671,-511){\makebox(0,0)[lb]{\smash{{\SetFigFont{12}{14.4}{\rmdefault}{\mddefault}{\updefault}{\color[rgb]{0,0,0}removed $0$-cubes}%
}}}}
\end{picture}%
	\caption{Legend for {\ref{fig:first}} and {\ref{fig:second}}}
\end{figure}

\begin{figure}[ht]
	\hskip-1cm\begin{subfigure}{0.3\textwidth}
		\includegraphics[width=1.0\textwidth]{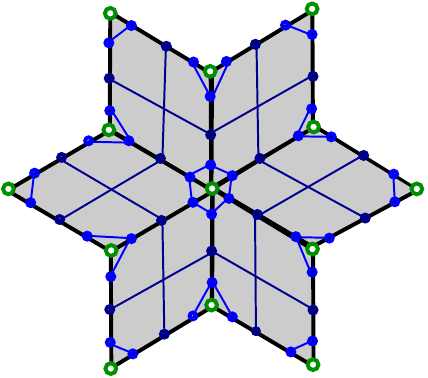}%
		\caption{}
	\end{subfigure}
	\hskip3cm
	\begin{subfigure}{0.3\textwidth}
		\includegraphics[width=0.8\textwidth]{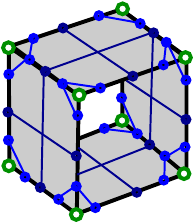}%
		\caption{}
	\end{subfigure}
	\caption{Two pieces of moment-angle complexes $X_K^{\geq 1}$}
	\label{fig:first}
\end{figure}

\begin{figure}[ht]
	\hskip-2cm\begin{subfigure}{0.3\textwidth}
		\includegraphics[width=1.0\textwidth]{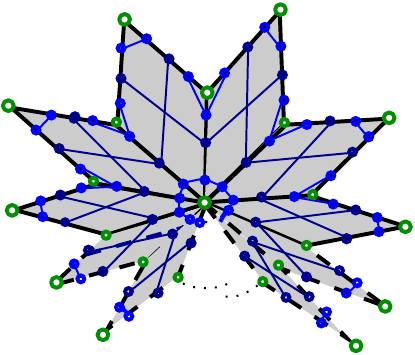}%
		\caption{}
	\end{subfigure}
	\hskip2cm
	\begin{subfigure}{0.3\textwidth}
		\includegraphics[width=1.4\textwidth]{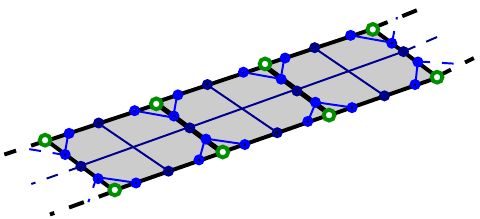}%
		\caption{}
	\end{subfigure}
	\caption{Universal cover of $X_K^{\geq 1}$ for the pieces from \ref{fig:first}}
	\label{fig:second}
\end{figure}

In any case the spectrum of the former contains that of the latter and the 
spectral hypothesis of \ref{cor:cubicalgarland} applies to the branched cover 
$Y$ only if it does to $X$. We conjecture a possible remedy for this 
after we discuss $k$-monodromy freeness of $k$-branched covers.

We claim that $k$-monodromy freeness of a $(k+1)$ dimensional cubical complex can be achieved with a $k$-branched 
cover of degree at most the order $2^{k+1}\,(k+1)!$ of the automorphism 
group of a $(k+1)$-cube.

Consider for a cubical complex $X$ of dimension $k+1$ with 
face poset $Q$ the bipartite graph 
$\Lambda=(\Lambda^{k+1}_0=Q^{k+1}, \Lambda^k_0=Q^k,\Lambda_1=\{(q^{k+1}>q^k)\}\subseteq Q^{k+1}\times Q^k)$ of the rank selected poset 
$Q^{k+1,k}$. There is a natural connection on this graph which chooses an 
isomorphism of the $(k+1)$-cubes $|Q_{\leq p}|$ and $|Q_{\leq q}|$ for any path in $\Lambda$ between vertices $p$ and $q$ in 
$\Lambda^{k+1}_0$ and hence if a basepoint $p\in\Lambda_0^{k+1}$ is fixed a map 
$\beta$ from loops in $\Lambda$ based at $p$ to the automorphism 
group $B_{k+1}$ of the $(k+1)$-cube  $|Q_{\leq p}|$.  
There is also the standard map $\pr$ from such loops to $\pr_1(X^{\geq k})$ 
and a group homomorphism $\nu:\pr_1(X^{\geq k})\rightarrow B_{k+1}$ with 
$\nu\circ \pr=\beta$.  For any subgroup $\tau$ of $\pr_1(X^{\geq k})$ there 
is an associated $k$-branched cover of $Q$ with degree the index of $\tau$ in 
$\pr_1(X^{\geq k})$ which is $k$-monodromy-free exactly if $\tau$ is 
contained in the kernel of $\nu$.

\begin{Conjecture}
  For every $k>0$ there is $d$ so that if $Q$ is a finite cubical poset of 
  dimension $k+1$ and every $k$-face is contained in at least $d$ facets 
  then there is a pair of $k$-branched covers $Q \leftarrow P \rightarrow R$ 
  so that the Garland poset of $R$ satisfies the hypotheses of \ref{cor:cubicalgarland}.
\end{Conjecture}

Note that the Alon-Boppona bound suggests that $d=\Omega(k^2)$ and 
its asymptotic sharpness that $d=\Theta(k^2)$.   


\subsection{Random simplicial and cubical complexes}

We define and consider a model of random $\Delta$-complexes 
$Y^\Delta_{h,d,k}$ and two related models of random cubical complexes 
$Y^\Box_{h,d,k}$  and $Z^\Box_{h,d,k}$ of dimension $k+1$ obtained by 
starting with $hd$ disjoint $(k+1)$-simplices or $(k+1)$-cubes and 
randomly identifying groups of $d$ parallel facets.

More notationally:
  Write $\Box^ {k+1} = [0,1]^{k+1}$ for the standard $(k+1)$-cube and 
	$$\Delta^{k+1}=\big\{\,t = (t_1,\ldots, t_{k+2}) \in \Box^{k+2}~\big|\,\sum_{K=1}^{k+2} t_K=1\,\big\}$$ for the standard $(k+1)$-simplex.  
Let us first describe the simplicial model.
Set 
	\begin{align*} Q^\Delta_{h,d,k}=\big\{\,  & q:[hd]\times[k+2]\rightarrow [h]\times[d] \,\big|\,\text{for all }   K\in[k+2] \\
		& \text{the restriction of } q \text{ to } [hd]\times\{K\} \text{ is bijective } \,\big\}.
  \end{align*}

  The set $[hd]$ indices $hd$ disjoint copies of $\Delta^{k+1}$
  and the set $[k+2]$ indices the $k+2$ facets of a $(k+1)$-simplex.
  If $q\in Q^\Delta_{h,d,k}$ write 
  $$X^\Delta_q=(\Delta^{k+1}\times[hd])\slash\sim_q$$ where 
    for $t = (t_1,\ldots, t_{k+2}) \in \Delta^ {k+1}$ we set
			$(t,F)\sim_q(t,F')$ if there is $K\in[k+2]$ with 
			$t_K=0$ and $q(F,K)_{\{1\}}=  q(F',K)_{\{1\}}$.  
  The condition that the restriction of $q \in Q^ \Delta_{h,d,k}$
  to $[hd]\times\{K\}$ is bijective for each $K \in [k+2]$
  is bijective implies that $d$ simplices are glued along each facet.
  In \ref{fig:exX} an example of $X^ {\Delta^ 2}_q$ is given for
  $h=d =2$ and $k=1$. Thus there are $2\cdot 2 = 4$ copies of 
  $\Delta^ 2$ with
  the values of a function $q$ specifies on the facets.
  The colors then indicate which identification have to be perfomred in order 
  to construct $X^\Delta_q$. 

  \begin{figure}
	  \centering
 \begin{tikzpicture}[
    every node/.style={font=\small},
    tri/.style={draw=black, fill=gray!35, thick}
]

\coordinate (A1) at (-3,0);
\coordinate (B1) at (-1,0);
\coordinate (C1) at (-2,1.6);
\filldraw[tri] (A1) -- (B1) -- (C1) -- cycle;
	 \node at ($(A1)!0.5!(B1)+(0,-0.25)$) {$q(1,1) = (1,1)$};
	 \node at ($(B1)!0.5!(C1)+(0.25,0.05)$) [right] {$\genfrac{}{}{0pt}{}{q(1,2)}{= (1,2)}$};
	 \node at ($(C1)!0.5!(A1)+(-0.25,0.05)$) [left] {$q(1,3) = (2,1)$};
\draw[green,very thick] (A1) -- (B1);
\draw[yellow,very thick] (A1) -- (C1);
\draw[brown,very thick] (B1) -- (C1);

\coordinate (A2) at (3,0);
\coordinate (B2) at (5,0);
\coordinate (C2) at (4,1.6);
\filldraw[tri] (A2) -- (B2) -- (C2) -- cycle;
	 \node at ($(A2)!0.5!(B2)+(0,-0.25)$) {$q(2,1)= (2,2)$};
	 \node at ($(B2)!0.5!(C2)+(0.25,0.05)$) [right] {$\genfrac{}{}{0pt}{}{q(2,2)}{=(1,1)}$};
	 \node at ($(C2)!0.5!(A2)+(-0.25,0.05)$) [left] {$q(2,3) = (1,2)$};
\draw[blue,very thick] (A2) -- (B2);
\draw[red,very thick] (A2) -- (C2);
\draw[brown,very thick] (B2) -- (C2);

\coordinate (A3) at (-3,-2.6);
\coordinate (B3) at (-1,-2.6);
\coordinate (C3) at (-2,-1);
\filldraw[tri] (A3) -- (B3) -- (C3) -- cycle;
	 \node at ($(A3)!0.5!(B3)+(0,-0.25)$) {$q(3,1)= (2,1)$};
	 \node at ($(B3)!0.5!(C3)+(0.25,0.05)$) [right]{$\genfrac{}{}{0pt}{}{q(3,2)}{=(2,1)}$};
	 \node at ($(C3)!0.5!(A3)+(-0.25,0.05)$) [left] {$q(3,3) = (2,2)$};
\draw[blue,very thick] (A3) -- (B3);
\draw[yellow,very thick] (A3) -- (C3);
\draw[violet,very thick] (B3) -- (C3);

\coordinate (A4) at (3,-2.6);
\coordinate (B4) at (5,-2.6);
\coordinate (C4) at (4,-1);
\filldraw[tri] (A4) -- (B4) -- (C4) -- cycle;
	 \node at ($(A4)!0.5!(B4)+(0,-0.25)$) {$q(4,1) = (1,2)$};
	 \node at ($(B4)!0.5!(C4)+(0.25,0.05)$) [right] {$\genfrac{}{}{0pt}{}{q(4,2)}{= (2,2)}$};
	 \node at ($(C4)!0.5!(A4)+(-0.25,0.05)$) [left] {$q(4,3)= (1,1)$};
\draw[green,very thick] (A4) -- (B4);
\draw[red,very thick] (A4) -- (C4);
\draw[violet,very thick] (B4) -- (C4);

\end{tikzpicture}
	  \caption{Example of $X^{\Delta}_q$ for $h = d = 2$} 
	  \label{fig:exX}
  \end{figure}

  Write $\tau:X^\Delta_q\rightarrow\Delta^{k+1}$ for the obvious projection.

  Similarly, we define models for cubical complexes.
  We set 
  {\small
  \begin{align*} 
	  Q^\Box_{h,d,k}=\big\{\, & q:[hd]\times[k+1]\times \{0,1\}\rightarrow [h]\times[d]\,\big|\\&  \bullet \text{for all } (K,E)\in [k+1]\times \{0,1\} \\
	  & ~\text{ the restriction of } q \text{ to } [hd]\times\{K\}\times\{E\} \text{ is bijective } \,\big\}
  \end{align*}
  }
  and 
  {\small
  \begin{align*} 
	  R^\Box_{h,d,k}=\big\{\, & r:[hd]\times[k+1]\times \{0,1\}\rightarrow [h]\times[2d] \,\big|\\& \bullet \text{for all } (H,K)\in [k]\times[k+1] \\ & ~\text{ the restriction of } q \text{ to } [hd]\times\{K\}\times\{0,1\} \text{ is bijective,} \\ & \bullet \big|\,\{(A,K,0)\,|\,r(A,K,0)_{\{1\}}=H\}\,\big|=d, \\ & \bullet r(B,K,0)_{\{1\}}\not= r(B,K,1)_{\{1\}} \text{ for all } (B,K)\in [hd]\times [k+1] \,\big\}
  \end{align*}
  }
  
  If $q\in Q^\Box_{h,d,k}$ or $R^\Box_{h,d,k}$ write $$X^\Box_q=(\Box^{k+1}\times[hd])\slash\sim_q$$ where
        for $t = (t_1,\ldots, t_{k+1}) \in \Box^ {k+1}$ we set
			$(t,F)\sim_q(t,F')$ if there is $K\in[k+1]$ with $t_K= E \in \{0,1\}$ and $q(F,K,E)_{\{1\}}=q(F',K,E)_{\{1\}}$.  
			Write $\tau$ for the map from $X^\Box_q$ to $\Box^{k+1}$ in the former case and to the torus $\mathbb{T}=\mathbb{T}^{k+1}$ obtained by identifying opposite facets of $\Box^{k+1}$ in the latter for the obvious projections.  

   Write $Y^\Delta_{h,d,k}$ and $Y^\Box_{h,d,k}$ for the uniform measures on the sets 
   $\big\{\,X^\Delta_q\,\big|\,q\in Q^\Delta_{h,d,k}\,\big\}$ and 
   $\big\{\,X^\Box_q\,\big|\,q\in Q^\Box_{h,d,k}\,\big\}$ respectively.   
   Write $Z^\Box_{h,d,k}$ for the measure on $\{X^\Box_q|q\in R^\Box_{h,d,k}\}$ weighted by one over the number of orientations of $\Gamma_{\gposet_kX^\Box_q}$ which are flows (indegree equal to outdegree).  
 
The following corollary uses the simplicial and the cubical Garland method 
together with spectral results by Pinsker \cite{P} and Friedman \cite{F} as 
in Bordenave \cite{B} which state that for uniformly chosen 
$d$-regular (bipartite) graphs the spectrum of the left normalized graph 
Laplacian is aas in 
$\{0,2\}\cup(1-\frac{2}{\sqrt{d}},1+\frac{2}{\sqrt{d}})$ or in the 
normalization of \cite{B}, Theorem 1, and \cite{BDH}, Theorem 3.2 (i), with 
$\epsilon=2(\sqrt{d}-\sqrt{d-1}-\delta$ the spectrum of the adjacency matrix is aas in $\{-d,d\}\cup[-2\sqrt{d}+\delta,2\sqrt{d}-\delta]$.  

	\medskip

	\begin{cor} ~~
		\begin{itemize}
			\item[(i)]
				If $d\geq 4(k+1)^2$ then $$\lim_{h\rightarrow\infty}\Pr_{X\in Y^\Delta_{h,d,k}}(\widetilde{\Hom}^k(X)=\{0\})=1.$$
		\item[(ii)] If $d\geq 4(2k+1)^2$ then
			\smallskip
			\begin{itemize}
				\item[(a)]
					$$\lim_{h\rightarrow\infty}\Pr_{X\in Y^\Box_{h,d,k}}(\widetilde{\Hom}^k(X)=\{0\})=1,$$
				\item[(b)]
					$$\lim_{h\rightarrow\infty}\Pr_{X\in Z^\Box_{h,d,k}}(\widetilde{\Hom}^k(X)=\QQ^{k+1})=1.$$
			\end{itemize}
			\end{itemize}
\end{cor}
\begin{proof}
Fix $X$ drawn from one of the three distributions and write $P$ for the face poset in the simplicial case or the nonempty face poset in the cubical ones and $\Gamma=\Gamma_{\gposet_kP}$ which is a regular graph of degree $d$ and decomposes into disjoint geometric link graphs indexed by the $(k-1)$-faces of the codomain of $\tau$ of which there are $\binom{k+2}{2}$, $4\binom{k+1}{2}$ and $\binom{k+1}{2}$ respectively and in the two cubical cases transversal link graphs of which there are $k+1$.  All of these graphs are bipartite except for the transversal ones for the $Z$ case which have no loops.  It is also easy to count the vertices of which there are $2h(k+2)$ and $2h(k+1)$ for the geometric ones and $2h(+1)$ and $h(k+1)$ for the transversal ones.  In every case the distributions are uniform either on bipartite regular graphs or in the case of the transversal components of $Z$ on loopless regular graphs so the required aas in $h$ spectral gaps follow from the cited spectral theorem.  
  For the cubical cases the $k$-monodromy free property is immediate by 
  construction since a choice of collar can just be pulled back using $\tau$.  

  Together these imply that \ref{cor:simplicialgarland} applies in the 
  $Y^\Delta_{h,d,k}$ case and \ref{cor:cubicalgarland} applies for the two cubical ones. 

  Note that the spectral conditions imply also that the associated graphs 
  are connected so the result holds for $Y^\Delta_{h,d,k}$.  

  It remains to show that in the $Y^\Box_{h,d,k}$ case the image $T_{k,P}$ 
  of the transversal classes is trivial while in the $Z^\Box_{h,d,k}$ case 
  the classes of the $k+1$ components are non-trivial and independent.

  For the former note that $T_k$ (and $L_k$) are natural and hence $T_{k,P}$ is the image under $\tau^\ast$ of $T_{k,\bigwedge^{k+1}}=0$.

  For the latter a similar argument either directly or by extending the notion of Garland posets to small categories in which every endomorphism and every isomorphism is an identity gives that $T_{k,P}$ is the image under $\tau^\ast$ of $\Hom^k(\TT^{k+1})\cong \QQ^{k+1}$ and it remains to show this map is injective.

For each $K\in[k+1]$ write $\gamma_K\in\Hom^k(\TT^{k+1})$ for the 
class supported cellularly by the $K$\textsuperscript{th} coordinate 
hyperplane.  Together these form a basis. Write $q|_K$ for the restriction of 
$q$ to $[hd]\times([k]-\{K\})\times\{0,1\}$ and 
$i^K:X^\Box_{q|_K}\rightarrow X^\Box_q$ for the (hyperplane) embedding 
taking each facet to the a meridian of a facet perpendicular to the 
$K$\textsuperscript{th} direction. The cellular chain $\omega_K$ with 
weight one on each cell is closed by the construction of $R^\Box$ and 
	the pairing 
	$\big(\,\tau^\ast(\gamma_{K'}),i^{K}_\ast(\omega_K)\,\big)$ yields
	$hd$ if $K=K'$ and zero otherwise.  

\end{proof}

In analogy with $Y^\Box_{h,d,k}$ and the cube, or $Z^\Box_{h,d,k}$ 
and the torus,
define for any pure $(k+1)$-dimensional cubical complex $\TT$ 
the uniform measure 
on $(k+1)$-dimensional and $d$-regular cubical complexes with a cubical map
$\tau$ to $\TT$ which is generically $hd$ to $1$.

\begin{Conjecture}
	If $d \geq 4(2k+1)^2$ then 
	in $h$ aas the induced map $\tau^*$ on the $k$\textsuperscript{th} cohomology groups 
	is an isomorphism to $\widetilde{\Hom}^ k(\TT)$.
\end{Conjecture}

\section*{Acknowledgments}
The authors thank the referee for their careful reading and many
helpful suggestions. In particular, we are grateful for pointing out
that \ref{cor:spectral} is implied by our results.  
The first author thanks Tali Kaufman for an interesting discussion
about the potential uses of the results from this paper.
We also thank Izhar Oppenheim for pointing out the need for 
transversal finiteness in \ref{cor:cubicalgarland}.

\end{document}